\documentclass{birkjour}
\usepackage{graphicx}
\usepackage{amssymb}
\usepackage{amsmath}
\usepackage{amsthm}
\usepackage{hyperref}
\usepackage{cleveref}
\usepackage{geometry}
\usepackage[utf8]{inputenc}
\usepackage[english]{babel}
\hypersetup{
	colorlinks=true,
	linkcolor=blue,
	citecolor=blue,
	filecolor=blue,
	urlcolor=blue,
}
\newtheorem{theorem}{Theorem}[section]

\newtheorem{definition}[theorem]{Definition}
\newtheorem{proposition}[theorem]{Proposition}

\newtheorem{lemma}[theorem]{Lemma}
\newtheorem{example}[theorem]{Example}
\newtheorem{corollary}[theorem]{Corollary}
\numberwithin{equation}{section}
\usepackage[mathlines]{lineno}
\modulolinenumbers[1]
\usepackage[mathlines]{lineno} 
\usepackage{amsmath}           
\usepackage{etoolbox}          
\renewcommand*\labelenumi{(\theenumi)} 
\newcommand*\linenomathpatch[1]{%
  \cspreto{#1}{\linenomath}%
  \cspreto{#1*}{\linenomath}%
  \csappto{end#1}{\endlinenomath}%
  \csappto{end#1*}{\endlinenomath}%
}
\newcommand*\linenomathpatchAMS[1]{%
  \cspreto{#1}{\linenomathAMS}%
  \cspreto{#1*}{\linenomathAMS}%
  \csappto{end#1}{\endlinenomath}%
  \csappto{end#1*}{\endlinenomath}%
}

\expandafter\ifx\linenomath\linenomathWithnumbers
  \let\linenomathAMS\linenomathWithnumbers
 
  \patchcmd\linenomathAMS{\advance\postdisplaypenalty\linenopenalty}{}{}{}
\else
  \let\linenomathAMS\linenomathNonumbers
\fi

\linenomathpatch{equation}
\linenomathpatchAMS{gather}
\linenomathpatchAMS{multline}
\linenomathpatchAMS{align}
\linenomathpatchAMS{alignat}
\linenomathpatchAMS{flalign}

\makeatletter
\patchcmd{\mmeasure@}{\measuring@true}{
  \measuring@true
  \ifnum-\linenopenaltypar>\interdisplaylinepenalty
    \advance\interdisplaylinepenalty-\linenopenalty
  \fi
  }{}{}
\makeatother

\begin{document}

\title[On stability of  restricted center properties and continuity of restricted center map$\ldots$]{On stability of  restricted center properties and continuity of restricted center map under $\ell_p$-direct sum}

\author{P. Gayathri}
\address{Department of Mathematical and Computational Sciences,\\ National Institute of Technology Karnataka, Surathkal, Mangaluru - 575 025, India}
\email[corresponding author]{pgayathri@nitk.edu.in}

\author{V. Thota}
\address{Department of Mathematics, National Institute of Technology Tiruchirappalli, \\ Tiruchirappalli, Tamil Nadu - 620 015, India}
\email{vamsinadh@nitt.edu}

\date{}

\begin{abstract}
	 We study the stability of various restricted center properties and certain continuity properties of the restricted center map. We observe that restricted center property, property-$(P_1)$ and semi-continuity properties of the restricted center map are preserved under $\ell_p$-direct sum $(1 \leq p <  \infty).$ It is shown that property-$(P_2)$ is stable under finite  $\ell_p$-direct sum, but not under infinite $\ell_p$-direct sum. Additionally, we introduce a notion called property-$(lP_2)$ as a sufficient condition for the continuity of the restricted center map. Further, the stability of property-$(lP_2)$ is established.
\end{abstract}
\maketitle

\section{Introduction}\label{sec:intr}
\enlargethispage{-1\baselineskip}
Let $X$ be a real Banach space. For any non-empty subset $M$ of $X,$ we denote the set of all non-empty closed subsets of $M$, non-empty closed bounded subsets of $M$, non-empty closed convex subsets of $M$ and  non-empty compact subsets of $M$ respectively by $CL(M)$, $CB(M)$, $CC(M)$ and $K(M).$ The closed unit ball and unit sphere of $X$ are denoted by $B_X$ and $S_X,$ respectively. 

For any $x\in X$ and $F \in CB(X),$ we define $r(x, F)=\sup\{\|x-a\| :  a \in F\}.$ The restricted radius of $F$ relative to $V\in CL(X),$ denoted by $rad_V(F),$ is defined as $rad_V(F)=\inf\{r(v, F): v \in V \}.$ The set $Z_V(F)$ of all points from $V$ at which the infimum in $rad_V(F)$ attain is called restricted centers of $F$ relative to $V,$ i.e., $Z_V(F)=\{v \in V : r(v, F)=rad_V(F)\}.$ 

We say that the pair $(V, F)$ has restricted center property (in short, r.c.p.) if $Z_V(F)$ is non-empty. Further, we say that the pair $(V, \mathcal{F})$ has r.c.p. if $(V,G)$ has r.c.p. for every $G \in \mathcal{F}.$ A sequence $(v_n)$ in $V$ is said to be a minimizing sequence for $F$ if $r(v_n, F) \to rad_V(F).$ For any $\mathcal{F} \subseteq CB(X),$ the map $Z_V$ from $(\mathcal{F}, \mathcal{H})$ to $(CB(V), \mathcal{H})$, defined as $F \mapsto Z_V(F)$, is said to be the restricted center map. Here, $\mathcal{H}$ is the Hausdorff metric on $CB(X)$ given by $\mathcal{H}(A,B)= \inf\{r>0: A \subseteq B+ rB_X \ \mbox{and} \ B \subseteq A+ rB_X \}$
for any $A, B \in  CB(X)$.

To study the continuity properties of the restricted center map, Mach \cite{Mach1980} introduced two notions called property-$(P_1)$ and property-$(P_2).$ To see some geometric aspects of these notions recently Lalithambigai et al. \cite{LPST2017} equivalently reformulated these notions as follows. For any $\delta \geq 0,$  we denote  $Z_V(F, \delta)=\{v \in V: r(v,F)\leq  rad_V(F)+ \delta\}.$ Clearly $Z_V(F,0)= Z_V(F).$ 

\begin{definition}
Let $V \in CL(X)$, $F\in CB(X)$ and $\mathcal{F} \subseteq CB(X).$
\begin{enumerate}

\item We say that the pair $(V, F)$ has property-$(P_1)$ if $(V, F)$ has r.c.p. and for every  $\epsilon>0$ there exists $\delta>0$ such that $Z_V(F, \delta) \subseteq Z_V(F)+\epsilon B_X.$ We say that the pair $(V, \mathcal{F})$ has property-$(P_1)$ if $(V, G)$ has property-$(P_1)$ for every $G \in \mathcal{F}.$ 

\item We say that the pair $(V, \mathcal{F})$ has property-$(P_2)$ if $(V, \mathcal{F})$ has r.c.p. and for every  $\epsilon>0$ there exists $\delta>0$ such that $Z_V(G, \delta) \subseteq Z_V(G)+\epsilon B_X$ for every  $G \in \mathcal{F}.$

\end{enumerate}
\end{definition}

In \cite[Theorem 5]{Mach1980}, Mach observed that  if $(V,\mathcal{F})$ has property-$(P_1)$, then the restricted center map $Z_V$ is upper Hausdorff semi-continuous on $(\mathcal{F},\mathcal{H}).$ Also, \cite[Theorem 6]{Mach1980}, it is shown that the restricted center map $Z_V$ is Hausdorff continuous on $(\mathcal{F}, \mathcal{H})$, whenever $(V,\mathcal{F})$ has property-$(P_2).$ However, the later result was recently improved, by Lalithambigai et al. \cite[Theorem 3.1]{LPST2017}, by showing that  property-$(P_2)$ leads to the uniform Hausdorff continuity of the restricted center map. It is easy to check that the converses of both results are not necessarily true in general, for instance see \cite[Example 3.2]{LPST2017}. At this stage, it is relevant to investigate the possibility of defining an intermediate notion that serves as a sufficient condition for the continuity of the restricted center map. To address this, in this article, we introduce a notion called property-$(lP_2)$. 

The theory of best approximation in Banach spaces provides a straightforward approach to explore these restricted center properties through their counterparts. Let  $V \in CL(X)$. If $F=\{x\}$ for some $x \in X$, then $rad_V(F)$ is the distance from $x$ to $V,$ denoted by $d(x,V)$, and $Z_V(F)$ is the set of all best approximants to $x$ from $V,$ i.e., $P_V(x)=\{v \in V: \|x-v\|=d(x,V)\}.$ Let $A$ be any non-empty subset of $X$ and $\mathcal{F}$ be the collection of all singleton subsets of $A.$ We say that the set $V$ is proximinal (respectively, strongly proximinal, uniformly strongly proximinal) on $A$ if $(V, \mathcal{F})$ has r.c.p. (respectively, property-$(P_1)$, property-$(P_2)$). Further, we say that the set $V$ is approximatively compact on $A$ if it is strongly proximinal on $A$ and $P_V(x)$ is compact for every $x \in A.$ The map $P_V$ from $A$ to $(CB(V), \mathcal{H})$ defined as $x \mapsto P_V(x)$, is called the metric projection. These proximinality notions are closely associated with some geometric properties and provide various continuity properties of the metric projection. We refer to \cite{DuST2017,InPr2017a} for more results on these notions.

 Let $1 \leq p < \infty.$ In \cite{LaNa2006}, Lalithambigai and Narayana observed the stability of the strong proximinality under finite $\ell_p$-direct sum. Further, in the same article, it is proved that the lower Hausdorff semi-continuity of the metric projection is stable under finite $\ell_{p}$-direct sum. The stability of approximative compactness under finite $\ell_{p}$-direct sum is proved by Bandyopadhyay et al. \cite{BLLN2008}. This result was later extended to infinite collection by Luo et al. in \cite{LuSZ2016}.

It is interesting to study whether analogous stability results hold for restricted center properties and for the continuity properties of restricted center map. Such a question is not new and has already been raised in the literature (see, \cite[Remark 12]{Rao2016}). In this paper, we present some affirmative answers to this question. We obtain some stability results of restricted center property, property-$(P_1)$, property-$(P_2)$ and the continuity properties of the restricted center map.

The article is organized as follows. The stability of the restricted center property under $\ell_p$-direct sum $(1\leq p <\infty)$ is shown in \Cref{sec:prel}. By obtaining some equivalent formulations of property-$(P_2)$, we demonstrate that the notions quasi uniform rotundity with respect to a subspace coincide with the property-$(P_2)$, which further equivalent to the  uniform Hausdorff continuity of the restricted center map. 

\Cref{sec:3} deals with the stability of properties-$(P_1)$ and $(P_2)$ under $\ell_p$-direct sum $(1 \leq p < \infty)$. We observe that property-$(P_1)$ is stable under infinite $\ell_p$-direct sum, but not property-$(P_2)$. However, we prove that property-$(P_2)$ is stable under finite $\ell_p$-direct sum. As a sufficient condition for the continuity of restricted center map, we introduce a notion called property-$(lP_2),$ which is intermediate between properties-$(P_1)$ and $(P_2).$ We notice that in spaces which are uniformly rotund in every direction property-$(P_1)$ and property-$(lP_2)$ coincide for closed convex sets. Also, we prove that in finite dimensional spaces the notions property-$(lP_2)$ and property-$(P_2)$ coincide for subspaces. Additionally, we present the stability of property-$(lP_2)$ by obtaining the stability of continuity properties of the restricted center map. 

\section{Preliminaries}\label{sec:prel}

In this section, we obtain the stability of restricted center property and relate property-$(P_2)$ with some geometric notions. These results are essential to prove our main results in the next section. Throughout the paper, we assume all subspaces to be closed. 
 
Let $I$ be any index set and  $\{X_i : {i\in I}\}$ be a collection of Banach spaces. For $1 \leq p < \infty,$ consider $X=(\oplus_pX_i)_{i \in I}$ and denote \[\mathcal{P}(X)=\{F=\Pi_{i \in I } F_i \in CB(X): F_i \in CB(X_i), i \in I\}.\]

The following restricted radius formula is essential to prove the stability of the restricted center property.

\begin{proposition}\label{prop:RadiusEquality_p}
Let $\{X_i: {i\in \mathbb{N}}\}$ be a collection of Banach spaces, $Y_i$ be a  subspace of $X_i$ for every $i\in \mathbb{N}$ and $1\leq p < \infty.$ Let $X=(\oplus_pX_i)_{i \in \mathbb{N}},$ $Y=(\oplus_pY_i)_{i \in \mathbb{N}}$ and $F=\Pi_{i \in \mathbb{N}}F_i \in \mathcal{P}(X).$ Then for any $x=(x_i)_{i \in \mathbb{N}} \in X,$ we have 
\begin{enumerate}
  \item $r(x,F) = \left(\sum_{i\in\mathbb{N}} r(x_i, F_i)^p\right)^{\frac{1}{p}}$;
   \item $rad_Y(F) = \left(\sum_{i\in\mathbb{N}} rad_{Y_i}(F_i)^p\right)^{\frac{1}{p}}$.
\end{enumerate}
\end{proposition}

\begin{proof}
$(1)$: Let $x=(x_i)_{i \in \mathbb{N}} \in X$. For every $i \in \mathbb{N},$ let $(z_{n,i})_{n \in \mathbb{N}}$ be a sequence in $F_i$ such that $\|x_i-z_{n,i}\| \to r(x_i, F_i).$ For every $n \in \mathbb{N},$ consider $z_n=(z_{n,i})_{i \in \mathbb{N}} \in F.$ Let  $m \in \mathbb{N}$. Since for any $n \in \mathbb{N}$
\[
\sum\limits_{i =1}^m \|x_i-z_{n,i}\|^p \leq \sum\limits_{i \in \mathbb{N}} \|x_i-z_{n,i}\|^p = \|x-z_n\|^p \leq  r(x, F)^p,
\]
we have $\sum_{i =1}^m r(x_i, F_i)^p \leq  r(x, F)^p.$ Thus, $\sum_{i \in \mathbb{N}} r(x_i, F_i)^p \leq  r(x, F)^p.$ Observe that for any $z=(z_i)_{i \in \mathbb{N}} \in F$, we have
\[ 
\|x-z\|^p = \sum_{i\in\mathbb{N}}\|x_i-z_i\|^p \leq \sum_{i\in\mathbb{N}} r(x_i, F_i)^p.
\]
Therefore, $r(x,F)^p \leq \sum_{i\in\mathbb{N}} r(x_i, F_i)^p.$ Hence the equality holds. \\
$(2)$: For any $y=(y_i)_{i \in \mathbb{N}}\in Y,$ by $(1),$ it follows that
\[ 
\sum\limits_{i \in \mathbb{N}} rad_{Y_i}(F_i)^p \leq \sum\limits_{i \in \mathbb{N}} r(y_i, F_i)^p = r(y, F)^p.
\]
Thus, $ \sum_{i \in \mathbb{N}} rad_{Y_i}(F_i)^p \leq rad_Y(F)^p.$ Now to prove the other inequality, for any $i \in \mathbb{N},$ choose a sequence $(y_{n,i})_{n \in \mathbb{N}}$  in $Y_i$ such that $r(y_{n,i}, F_i) \to rad_{Y_i}(F_i).$ Let $m \in \mathbb{N}$. For every $n \in \mathbb{N},$ define $z_n = (y_{n,1}, y_{n,2},\ldots,y_{n,m},0,0,\ldots) \in Y.$ Since, by $(1),$
\[
rad_Y(F)^p \leq r(z_n, F)^p = \sum\limits_{i =1}^m r(y_{n,i}, F_i)^p + \sum\limits_{i > m} r(0, F_i)^p
\]
holds for any $n \in \mathbb{N},$ we have $rad_Y(F)^p \leq  \sum_{i =1}^m rad_{Y_i}(F_i)^p + \sum_{i > m} r(0, F_i)^p .$ Thus, $rad_Y(F)^p \leq \sum_{i \in \mathbb{N}} rad_{Y_i}(F_i)^p.$ Hence the equality holds. 
\end{proof}

\begin{theorem}\label{prop: rcp l_p}
Let $\{X_i : {i\in \mathbb{N}}\}$ be a collection of Banach spaces, $Y_i$ be a  subspace of $X_i$ for every $i\in \mathbb{N}$ and $1\leq p < \infty.$ Let $X=(\oplus_pX_i)_{i \in \mathbb{N}}$, $Y=(\oplus_pY_i)_{i \in \mathbb{N}}$ and $F= \Pi_{i \in \mathbb{N}}F_i \in \mathcal{P}(X).$  Then  $(Y_i, F_i)$ has r.c.p. for every $i \in \mathbb{N}$ if and only if $(Y, F)$ has r.c.p.  Further, we have $Z_Y(F)=(\oplus_pZ_{Y_i}(F_i))_{i \in \mathbb{N}}= \Pi_{i \in \mathbb{N}}Z_{Y_i}(F_i).$
\end{theorem}

\begin{proof}
Let $(Y_i, F_i)$ has r.c.p. for every $i \in \mathbb{N}.$ Choose $y_i \in Z_{Y_i}(F_i)$ for every $i \in \mathbb{N}$ and consider $y=(y_i)_{i \in \mathbb{N}}.$ First we prove that $y \in Y.$ For this let $z=(z_i)_{i\in\mathbb{N}} \in F.$ Since for every $i \in \mathbb{N},$ we have $\|y_i-z_i\| \leq r(y_i, F_i)$ which implies $\| y_i\| \leq rad_{Y_i}(F_i)+\|z_i\|.$ Thus, it follows from \Cref{prop:RadiusEquality_p},
\[
\left(\sum_{i \in \mathbb{N}} \|y_i\|^p\right)^\frac{1}{p} \leq  \left(\sum_{i \in \mathbb{N}} rad_{Y_i}(F_i)^p\right)^\frac{1}{p}+ \left(\sum_{i \in \mathbb{N}}\|z_i\|)^p\right)^\frac{1}{p} = rad_Y(F)+\|z\|
\]
and hence $y \in Y.$ Further, by \Cref{prop:RadiusEquality_p}, $y \in Z_Y(F).$ Thus, $(Y,F)$ has r.c.p. From the preceding argument, observe that $\Pi_{i \in \mathbb{N}}Z_{Y_i}(F_i)=(\oplus_p Z_{Y_i}(F_i))_{i \in \mathbb{N}}  \subseteq Z_Y(F).$ To prove the converse, let $(Y, F)$ has r.c.p. and $w=(w_i)_{i \in \mathbb{N}} \in Z_Y(F).$ We claim that $w_i \in Z_{Y_i}(F_i)$ for every $i \in \mathbb{N}.$ Suppose there exists $j \in \mathbb{N}$ such that $w_j \notin Z_{Y_j}(F_j).$ Since $rad_{Y_j}(F_j) < r(w_j, F_j)$  and $rad_{Y_i}(F_i) \leq r(w_i, F_i)$ for every $i \in \mathbb{N},$ by \Cref{prop:RadiusEquality_p}, we have $rad_Y(F) < r(w, F).$ This is a contradiction. Therefore, $w_i \in Z_{Y_i}(F_i)$ for every $i \in \mathbb{N}.$ Hence, $(Y_i, F_i)$ has r.c.p. for every $i \in \mathbb{N},$ further it follows that $Z_Y(F) \subseteq (\oplus_p Z_{Y_i}(F_i))_{i \in \mathbb{N}}.$
\end{proof}

\begin{corollary}\label{cor:rcp_uniq}
Let $\{X_i : {i\in \mathbb{N}}\}$ be a collection of Banach spaces, $Y_i$ be a  subspace of $X_i$ for every $i\in \mathbb{N}$ and $1\leq p < \infty.$ Let $X=(\oplus_pX_i)_{i \in \mathbb{N}}$, $Y=(\oplus_pY_i)_{i \in \mathbb{N}}$ and $F= \Pi_{i \in \mathbb{N}}F_i \in \mathcal{P}(X).$  Then  $Z_Y(F)$ is singleton if and only if $Z_{Y_i}(F_i)$ is singleton for every $i \in \mathbb{N}.$
\end{corollary}

It is easy to verify that for any subspace $Y$ of $X$, if $(Y, \mathcal{F})$ has r.c.p. (respectively, property-$(P_1)$) where $\mathcal{F}=\{F \in CB(X): rad_Y(F)=1\}$, then $(Y, CB(X))$ has r.c.p. (respectively, property-$(P_1)$). We need this fact and the following equivalences of property-$(P_2)$ to prove our main results in the next section.  

  \begin{proposition}\label{prop:P_2Equi}
    Let $Y$ be a subspace of $X$ with $(Y, CB(X))$ has r.c.p. and $\alpha>0$. Then the following statements are equivalent.
    \begin{enumerate}
        \item $(Y, \mathcal{F}_1)$ has property-$(P_2)$, where $\mathcal{F}_1= \{F \in CB(X): rad_Y(F)= r(0,F)=1\}.$
        \item  $(Y, \mathcal{F}_2)$ has property-$(P_2)$, where $\mathcal{F}_2= \{F \in CB(X): rad_Y(F)\leq \alpha\}.$
        \item  $(Y, \mathcal{F}_3)$ has property-$(P_2)$, where $\mathcal{F}_3= \{F \in CB(X): rad_Y(F)=\alpha\}.$
    \end{enumerate}
\end{proposition}

\begin{proof}
    $(1)\Rightarrow (2)$: Let $\epsilon >0.$ By $(1),$ there exists $0<\delta <1$ such that $Z_Y(G,\delta) \subseteq Z_Y(G)+\frac{\epsilon}{\alpha}B_X$ for every $G \in \mathcal{F}_1.$ Now for  $\delta_1 = \frac{\epsilon\delta}{3},$ we show that $Z_Y(F, \delta_1) \subseteq Z_Y(F)+\epsilon B_X$ for every $F\in\mathcal{F}_2.$  Let $F \in \mathcal{F}_2$. \\
    Case-(i): Let $ rad_Y(F) \leq \frac{\epsilon}{3}.$ Then for any $v \in Z_Y(F, \delta_1)$ and $w \in Z_Y(F)$ we have
    \[\|v-w\| \leq r(v, F)+r(w, F) \leq rad_Y(F)+\delta_1+rad_Y(F) < \epsilon.\]
    Thus, $Z_Y(F,\delta_1) \subseteq Z_Y(F)+\epsilon B_X.$ \\
    Case-(ii): Let $\frac{\epsilon}{3}< rad_Y(F) \leq \alpha.$ Consider $G=\frac{1}{rad_Y(F)}(F-w),$ for some $w \in Z_Y(F).$ Observe that $rad_Y(G)=r(0, G)=1.$ Thus, by assumption, we have $Z_Y(G,\delta) \subseteq Z_Y(G)+\frac{\epsilon}{\alpha}B_X.$ For any $\eta\geq 0$, it is easy to verify that
    \[ Z_Y(G, \eta) = \frac{1}{rad_Y(F)} \left(Z_Y(F, rad_Y(F)\eta)-w \right).\]
    Therefore, we have $Z_Y(F, rad_Y(F)\delta) \subseteq Z_Y(F)+rad_Y(F)\frac{\epsilon}{\alpha}B_X$ and hence it follows that $Z_Y(F,\delta_1) \subseteq Z_Y(F)+\epsilon B_X.$\\
    $(2)\Rightarrow(3)$: Obvious.\\
    $(3)\Rightarrow(1)$: Let $\epsilon>0.$ By (3), there exists $\delta>0$ such that $Z_Y(G,\alpha\delta)\subseteq Z_Y(G)+\alpha\epsilon B_X$ for every $G \in \mathcal{F}_3.$ Let $F\in \mathcal{F}_1.$ Since $\alpha F \in \mathcal{F}_3,$  we have $Z_Y(\alpha F, \alpha \delta) \subseteq Z_Y(\alpha F)+\alpha \epsilon B_X$, which further leads to  $Z_Y(F, \delta) \subseteq Z_Y(F)+\epsilon B_X.$ Hence, $(Y, \mathcal{F}_1)$ has property-$(P_2).$
\end{proof}

Now, we present some results which exhibit equivalent relations between property-$(P_2)$ and geometric properties.  

As a sufficient condition for the existence of restricted centers, Calder et al. \cite{CaCH1973} introduced the following notion. The closed ball centered at $x \in X$ with radius $r>0$ is denoted by $B[x,r]$.   

\begin{definition}\cite{CaCH1973}
    Let $Y$ be a subspace of $X$. We say $X$ is quasi uniformly rotund with respect to $Y$, if for every $\epsilon>0$ there exists $\delta>0$ such that for every $v \in Y$, there exists $w \in Y$ such that $ w \in B[0, \epsilon]$ and $B[0,1] \cap B[v,1-\delta] \subseteq B[w, 1-\delta]$.
\end{definition}

For more results on quasi uniform rotundity in terms of restricted centers, we refer to \cite{LPST2017,Rao2016,Vese2017}. In \cite[Lemma 3.5]{LPST2017}, it is shown that property-$(P_2)$ is a necessary condition for quasi uniform rotundity. However,  in \Cref{lem: QUR}, we show that the converse also holds through the continuity of the restricted center map. For this, we recall the following continuity definitions.  

\begin{definition}
Let $(N,d)$ be a metric space, $X$ be a Banach space and   $T: (N,d) \to (CB(X), \mathcal{H})$ be a set valued map. Then the map $T$ is said to be 
\begin{enumerate}
    \item upper Hausdorff semi-continuous (in short, uHsc) on $N$ if for every  $a_0 \in N$ and $\epsilon>0$ there exists $\delta>0$ such that $T(a) \subseteq T(a_0)+ \epsilon B_X$, whenever $a \in N$ and $d(a, a_0) < \delta$;
   
    \item  lower Hausdorff semi-continuous (in short, lHsc) on $N$ if for every $a_0 \in N$ and  $\epsilon>0$ there exists $\delta>0$ such that $T(a_0) \subseteq T(a)+ \epsilon B_X$, whenever  $a \in N$ and $d(a, a_0) < \delta$;
    
    \item Hausdorff continuous on $ N$ if it is both uHsc and lHsc on $N$;
    
    \item uniformly Hausdorff continuous on $N$ if  for every $\epsilon>0$ there exists $\delta>0$ such that $\mathcal{H}(T(a), T(b))< \epsilon$, whenever  $a,b \in N$ with $d(a, b) < \delta$.
\end{enumerate}
\end{definition}

In the above definition, if $T(a)$ is a singleton set for every $a \in N,$ then upper Hausdorff semi-continuity, lower Hausdorff semi-continuity and Hausdorff continuity of $T$ are equivalent on $N$. 

\begin{theorem}\label{lem: QUR}
    Let $Y$ be a subspace of $X$ and $\mathcal{F}=\{F \in CB(X): rad_{Y}(F) \leq \alpha\}$ for some $\alpha>0.$ Then the following statements are equivalent. 
    \begin{enumerate}
        \item $X$ is quasi uniformly rotund with respect to $Y$.
        
        \item $(Y, \mathcal{F})$ has property-$(P_2)$.
        
        \item The restricted center map $Z_Y$ is non-empty valued and  uniformly Hausdorff continuous on $(\mathcal{F}, \mathcal{H}).$ 
    \end{enumerate}
\end{theorem}

\begin{proof}
    $(1) \Rightarrow (2)$: This implication follows from \cite[Lemma 3.5]{LPST2017} and \Cref{prop:P_2Equi}.\\
    $(2) \Rightarrow (3)$: This implication follows from \cite[Theorem 3.1]{LPST2017}.\\
    $(3) \Rightarrow (1)$: This implication  follows from \cite[Theorem 3.10 and Remark 3.12]{Vese2017}.
\end{proof}

Now, we present a characterization of uniform rotundity with respect to a subspace in terms of property-$(P_2)$, which is not explicitly stated in the literature; however, it follows from the results of \cite{AmZi1980,LPST2017}, and this characterization will be utilized in \Cref{sec:3}.

\begin{definition}\label{UR_wrt}\cite{AmZi1980}
    Let $Y$ be a subspace of $X$. We say $X$ is  uniformly rotund with respect to $Y$ if $(x_n)$ and $(y_n)$ are two sequences  in $S_X$ such that $x_n-y_n \in Y$ for every $n \in \mathbb{N}$ and $\left\Vert \frac{x_n+y_n}{2} \right\Vert \to 1,$  it follows that   $x_n-y_n \to 0$.
\end{definition}

\begin{theorem}\label{lem: UR w.r.t. Y}
    Let $Y$ be a subspace of $X$ and $\alpha>0$. Then the following statements are equivalent.
    \begin{enumerate}
        \item $X$ is uniformly rotund with respect to $Y.$
        
        \item $(Y, \mathcal{F}_1)$ has property-$(P_2)$ and $Z_Y(F)$ is  singleton for every $F \in \mathcal{F}_1$, where $\mathcal{F}_1=\{F \in CB(X): rad_Y(F) \leq \alpha\}$.
        
        \item $(Y, \mathcal{F}_2)$ has property-$(P_2)$ and $Z_Y(F)$ is  singleton for every $F \in \mathcal{F}_2$, where $\mathcal{F}_2=\{F \in K(X): rad_Y(F) \leq \alpha\}$.
    \end{enumerate}
    \end{theorem}
    
    \begin{proof}
    $(1) \Leftrightarrow (2)$: These implications follow from  \cite[Theorem 3.9]{LPST2017}.\\
    $(2) \Rightarrow (3)$: Obvious.\\
    $(3) \Rightarrow (1)$: Since, by \cite[Theorem 3.1]{LPST2017}, the restricted center map $Z_Y$ is uniformly Hausdorff continuous on $(\mathcal{F}_2, \mathcal{H}),$ it follows from the assumption and \cite[Theorem 1.11]{AmZi1980}, that $X$ is uniformly rotund with respect to $Y$.
\end{proof}

\section{Restricted center properties under $\ell_p$-direct sum}\label{sec:3}

This section deals with the stability of property-$(P_1),$ property-$(P_2)$  and continuity of the restricted center map under $\ell_p$-direct sum $(1 \leq p < \infty).$ We need the following sequential characterizations of properties-$(P_1)$ and $(P_2)$ which are easy to verify.

\begin{proposition}\label{prop:p1p2}
    Let $V \in CL(X)$, $F \in CB(X)$ and $\mathcal{F} \subseteq CB(X)$. Then the following statements hold.
    \begin{enumerate}
        \item $(V, F)$ has property-$(P_1)$ if and only if $(V, F)$ has r.c.p. and for every minimizing  sequence $(x_n)$ for $F$ in $V$, there exists a sequence $(y_n)$ in $V$ such that  $y_n \in Z_V(F)$ for every $n \in \mathbb{N}$ satisfying $\|x_n-y_n\| \to 0.$

        \item $(V, \mathcal{F})$ has property-$(P_2)$ if and only if $(V, \mathcal{F})$ has r.c.p. and  whenever $(x_n)$ is a  sequence in $V$, $(F_n)$ is a  sequence in  $\mathcal{F}$ with $r(x_n,F_n)- rad_V(F_n) \to 0$, there exists a sequence $(y_n)$ in $V$ such that $y_n \in Z_V(F_n)$ for every $n \in \mathbb{N}$ satisfying $\|x_n-y_n\| \to 0.$
    \end{enumerate}
\end{proposition}

To see the stability of property-$(P_1)$, we need the following lemma, which is a generalization of \cite[Lemma 3]{LuSZ2016}.

\begin{lemma}\label{lem:unimini}
    Let $\{X_i : {i\in \mathbb{N}}\}$ be a collection of Banach spaces, $Y_i$ be a  subspace of $X_i$ for every $i\in \mathbb{N}$ and $1\leq p <\infty.$ Let $X=(\oplus_pX_i)_{i \in \mathbb{N}},$ $Y=(\oplus_pY_i)_{i \in \mathbb{N}},$ $F=\Pi_{i \in \mathbb{N}}F_i \in \mathcal{P}(X)$ and $(y_n)$ be a minimizing sequence for $F$ in $Y$, where $y_n=(y_{n,i})_{i \in \mathbb{N}}$ for every $n \in \mathbb{N}.$ Then for every $\epsilon >0$ there exists $j \in \mathbb{N}$ such that $\sum_{i>j}\|y_{n,i}\|^p < \epsilon^p$ for all $n \in \mathbb{N}.$
\end{lemma}

\begin{proof}
Suppose there exists $\epsilon >0$ such that for every $j \in \mathbb{N}$ there exists a subsequence $(n_m)$, depending on $j$, such that $\sum_{i>j}\|y_{n_m, i}\|^p \geq \epsilon^p$ for all $m \in \mathbb{N}.$ Since, by \Cref{prop:RadiusEquality_p}, $r(0, F)^p = \sum_{i \in \mathbb{N}} r(0, F_i)^p,$ there exists $j_0 \in \mathbb{N}$ such that $\sum_{i>j_0} r(0, F_i)^p < (\frac{\epsilon}{3})^p.$ By preceding argument, for this ‘$j_0$’ there exists a subsequence $(n_k)$ such that $\sum_{i>j_0} \|y_{n_k, i}\|^p \geq \epsilon^p$ for all $k \in \mathbb N.$
For any $k \in \mathbb N,$ by \Cref{prop:RadiusEquality_p}, we have
\begin{align*}
    r(y_{n_k}, F)^p  & = \left(\sum_{i \leq j_0} r(y_{n_k,i}, F_i)^p +\sum_{i>j_0} r(0, F_i)^p \right) -\sum_{i>j_0} r(0, F_i)^p+ \sum_{i >j_0} r(y_{n_k,i}, F_i)^p \\
    & > rad_Y(F)^p - \left(\frac{\epsilon}{3}\right)^p +  \left[\left(\sum_{i>j_0} \|y_{n_k,i}\|^p\right)^{\frac{1}{p}}-\left(\sum_{i>j_0} r(0, F_i)^p\right)^{\frac{1}{p}}\right]^p \\
    & > rad_Y(F)^p - \left(\frac{\epsilon}{3}\right)^p +\left(\epsilon-\frac{\epsilon}{3}\right)^p\\
    & \geq rad_Y(F)^p + \left(\frac{\epsilon}{3}\right)^p.
\end{align*}
This is a contradiction. Hence the proof. 
\end{proof}

\begin{theorem}\label{thm:pds_P1}
    Let $\{X_i: {i\in \mathbb{N}}\}$ be a collection of Banach spaces, $Y_i$ be a  subspace of $X_i$ for every $i\in \mathbb{N}$ and $1\leq p < \infty.$ Let $X=(\oplus_pX_i)_{i \in  \mathbb{N}},$ $Y=(\oplus_pY_i)_{i \in  \mathbb{N}}$ and $F=\Pi_{i \in \mathbb{N}}F_i \in \mathcal{P}(X).$ Then $(Y_i, F_i)$ has property-$(P_1)$ for every $i\in \mathbb{N}$ if and only if  $(Y, F)$ has property-$(P_1).$ 
\end{theorem}

\begin{proof}
Let $(Y_i, F_i)$ has property-$(P_1)$ for every $i\in \mathbb{N}$. By \Cref{prop: rcp l_p}, $(Y, F)$ has r.c.p. Let $(y_n)$ be a minimizing sequence for $F$ in $Y$, where $y_n=(y_{n,i})_{i \in \mathbb{N}}$ for all $n \in \mathbb{N}.$ Let $i \in \mathbb{N}$. By \Cref{prop:RadiusEquality_p}, we have  
\begin{align*}
    rad_{Y_i}(F_i)^p &\leq r(y_{n, i}, F_i)^p \\
    & =  r(y_n, F)^p -\sum_{j \in \mathbb{N} \backslash \{i\}}r(y_{n,j}, F_j)^p \\
    & \leq  r(y_n, F)^p -\sum_{j \in \mathbb{N} \backslash \{i\}} rad_{Y_j}(F_j)^p + rad_Y(F)^p-rad_Y(F)^p \\
    & \leq r(y_n, F)^p-rad_Y(F)^p+ rad_{Y_i}(F_i)^p.
\end{align*}
Thus, $r(y_{n,i}, F_i) \to rad_{Y_i}(F_i).$ Since $(Y_i, F_i)$ has property-$(P_1),$ by \Cref{prop:p1p2},  there exists $z_{n,i} \in Z_{Y_i}(F_i)$ for every $n \in \mathbb{N}$ such that $\|y_{n, i}-z_{n, i}\| \to 0$ as $n \to \infty.$ By \Cref{prop: rcp l_p}, we have $z_n=(z_{n,i})_{i\in \mathbb N} \in Z_Y(F)$ for every $n \in \mathbb{N}.$ Now, by \Cref{prop:p1p2}, it is enough to prove that $\|y_n-z_n\|\to 0.$ To see this let $\epsilon >0.$ By \Cref{lem:unimini}, there exists $j \in \mathbb N$ such that $\sum_{i>j}\|y_{n,i}\|^p < \epsilon^p$ and $\sum_{i>j}\|z_{n,i}\|^p < \epsilon^p$ for all $n \in \mathbb{N}$. Also, there exists $n_0 \in \mathbb N$ such that $\sum_{i \leq j}\|y_{n, i}-z_{n,i}\|^p \leq (2\epsilon)^p$ for all $n \geq n_0.$ Therefore for any $n  \geq n_0$, we have
\begin{align*}
    \|y_n-z_n\|^p & = \sum_{i \leq j}\|y_{n,i}-z_{n,i}\|^p + \sum_{i>j} \|y_{n,i}-z_{n,i}\|^p \\
    & \leq (2\epsilon)^p + \left(\left(\sum_{i > j}\|y_{n,i}\|^p\right)^\frac{1}{p}+ \left(\sum_{i > j}\|z_{n,i}\|^p\right)^\frac{1}{p}\right)^p\\
    & < (2\epsilon)^p + (2\epsilon)^p.
\end{align*}
Thus, $\|y_n-z_n\| \to 0.$ To prove the converse, let $(Y, F)$ has property-$(P_1)$. It follows from \Cref{prop: rcp l_p}  that $(Y_i, F_i)$ has r.c.p for every $i \in \mathbb{N}$. Let $j \in \mathbb{N}$ and $(y_{n,j})_{n \in \mathbb{N}}$ be a minimizing sequence for $F_{j}$ in $Y_{j}$. For every $i \in \mathbb{N}$, choose  $y_i \in Z_{Y_i}(F_i)$. Now  for every $n \in \mathbb{N},$ consider
$y_{n,i}= y_i$ for every $ i \in \mathbb{N} {\setminus} \{j\}$ and $z_n=(y_{n,i})_{i \in \mathbb{N}} \in Y.$ By \Cref{prop:RadiusEquality_p}, it follows that
\[
rad_Y(F)^p \leq r(z_n,F)^p = \sum_{i \in \mathbb{N}} r(y_{n,i}, F_i)^p = \sum_{i \in \mathbb{N} {\setminus} \{j\}} rad_{Y_i}(F_i)^p+ r(y_{n, j}, F_{j})^p
\]
for all $n \in \mathbb{N}$ and hence  $r(z_n,F) \to rad_Y(F).$ Since $(Y, F)$ has property-$(P_1),$ by \Cref{prop:p1p2}, there exists $w_n=(w_{n,i})_{i \in \mathbb{N}} \in Z_Y(F)$ for every $n \in \mathbb{N}$ such that $\|z_n-w_n\| \to 0$, which further implies $\|y_{n,j}-w_{n,j}\| \to 0$. Note that, by \Cref{prop: rcp l_p}, $w_{n,j} \in Z_{Y_{j}}(F_{j})$ for all $n \in \mathbb{N}.$ Thus, $(Y_{j}, F_{j})$ has property-$(P_1)$. Hence the proof. 
\end{proof}

The following corollary is a direct consequence of \Cref{thm:pds_P1}.

\begin{corollary}\label{cor:pds_sp}
 Let $\{X_i: {i\in \mathbb{N}}\}$ be a collection of Banach spaces, $Y_i$ be a  subspace of $X_i$ for every $i\in \mathbb{N}$ and $1\leq p < \infty.$ Let $X=(\oplus_pX_i)_{i \in \mathbb{N}}$ and $Y=(\oplus_pY_i)_{i \in \mathbb{N}}.$ Then  $Y_i$ is strongly proximinal on $X_i$ for every $i\in \mathbb{N}$ if and only if $Y$ is strongly proximinal on $X$. 
\end{corollary}

For any $V \in CL(X)$ and $\mathcal{F} \subseteq CB(X),$ we say that the pair $(V, \mathcal{F})$ is well posed (respectively, generalized well posed) if every minimizing sequence for $F \in \mathcal{F}$ in $V$ converges (respectively, has a convergent subsequence). The subsequent result on the stability of well posed-ness and generalized well posed-ness follows using \Cref{prop: rcp l_p} and a similar technique involved in the proof of \Cref{thm:pds_P1}.

\begin{corollary}
    Let $\{X_i: {i\in \mathbb{N}}\}$ be a collection of Banach spaces, $Y_i$ be a  subspace of $X_i$ for every $i\in \mathbb{N}$ and $1\leq p < \infty.$ Let $X=(\oplus_pX_i)_{i \in  \mathbb{N}}$ and  $Y=(\oplus_pY_i)_{i \in  \mathbb{N}}$. Then $(Y_i, CB(X_i))$ is well posed (respectively, generalized well posed) for every $i\in \mathbb{N}$ if and only if  $(Y, \mathcal{P}(X))$ is well posed (respectively, generalized well posed).
\end{corollary}

Now, we proceed to discuss the stability of property-$(P_2)$. It is easy to verify from the proof of \Cref{thm:pds_P1}  that if $(Y, \mathcal{F})$ has property-$(P_2)$, where $\mathcal{F}=\{F \in \mathcal{P}(X): rad_Y(F)=1\}$, then  $(Y_i, \mathcal{F}_i)$ has property-$(P_2)$, where $\mathcal{F}_i=\{F_i \in CB(X_i): rad_{Y_i}(F_i)=1\}$ for every $i\in \mathbb{N}$. However, the following example reveals that, for any $1 \leq p < \infty$, the converse does not hold for infinite collection. Nevertheless, in the subsequent result, we prove that the converse is true when the collection is finite.

\begin{example}\label{eg: p2}
    Let $1\leq p <  \infty$. For every $i \in \mathbb{N}$, let $X_i= (\mathbb{R}^3, \|\cdot\|_{p_i})$, where $p_i=1+\frac{1}{i}$ and  $Y_i=span\{(1,0,0), (0,1,0)\}$ be the subspace of $X_i.$ Consider $X=(\oplus_p X_i)_{i \in \mathbb{N}}$ and $Y=(\oplus_p Y_i)_{i \in \mathbb{N}}.$ Since each $X_i$ is uniformly rotund, by \cite[Theorem 3.8]{LPST2017}, it follows that $(Y_i, \mathcal{F}_i)$ has property-$(P_2)$ for every $i \in \mathbb{N}$, where $\mathcal{F}_i=\{F_i \in CB(X_i): rad_{Y_i}(F_i)=1\}.$ Now we prove that $(Y, \mathcal{F})$ does not have property-$(P_2)$, where $\mathcal{F}=\{F\in \mathcal{P}(X): rad_Y(F)=1\}.$ Let $n \in \mathbb{N}.$ Consider $x_n= (x_{n,i})_{i \in \mathbb{N}} \in Y$ and $y_n= (y_{n,i})_{i \in \mathbb{N}} \in Y$ where
    \begin{align*}
        x_{n,i}= \begin{cases}(0,0,0), & \text{if } i\neq n ;\\  (1,0,0), & \text{if } i=n,
\end{cases}& \hspace{2cm} y_{n,i}= \begin{cases}(0,0,0), & \text{if } i\neq n ;\\  (0,1,0), & \text{if } i=n
\end{cases}
    \end{align*}
    and $G_n=\left\{\frac{x_n+y_n}{\|x_n+y_n\|}, -\frac{x_n+y_n}{\|x_n+y_n\|}\right\} \in CB(X).$
    Since $G_n= \Pi_{i \in \mathbb{N}} G_{n,i}$ where 
    \[
G_{n,i}=\begin{cases}\{(0,0,0)\}, & \text{if } i\neq n ;\\ \left\{ \frac{(1,1,0)}{\|(1,1,0)\|_{p_n}}, \frac{-(1,1,0)}{\|(1,1,0)\|_{p_n}}\right\}, & \text{if } i=n,
\end{cases}
\]
we have $G_n \in \mathcal{P}(X).$ Observe that $rad_Y(G_{n})=r(0, G_n)= 1.$ Define  $w_n=\frac{x_n-y_n}{\|x_n+y_n\|} \in Y.$     Since 
    \[
    rad_Y(G_n) \leq r(w_n,G_n) = \max\left\{\left\Vert w_n+\frac{x_n+y_n}{\|x_n+y_n\|}\right\Vert, \left\Vert w_n-\frac{x_n+y_n}{\|x_n+y_n\|}\right\Vert \right\}=\frac{2}{\|x_n+y_n\|}
    \]
    for all $n \in \mathbb{N}$ and $\left\Vert\frac{x_n+y_n}{2}\right\Vert \to 1$, we have $r(w_n, G_n) -rad_Y(G_n)  \to 0$. Now, by \Cref{prop:p1p2}, it is enough to prove that $d(w_n, Z_Y(G_n))$ does not converge to $0$. Fix $n \in \mathbb{N}$.  Since for any $i \in \mathbb{N}$, $X_i$ is uniformly rotund, by \cite{Gark1962}, it follows that $Z_{Y_i}(G_{n,i})$ is singleton. Observe that  $Z_{Y_i}(G_{n,i})=\{0\}$ for all $i \in \mathbb{N}$ further, using \Cref{prop: rcp l_p}, we have  $Z_Y(G_n)=\{0\}.$ Thus,
    \[
 d(w_n, Z_Y(G_n))=\|w_n-0\|= \left\Vert \frac{x_n-y_n}{\|x_n+y_n\|} \right\Vert=1.
    \]
   Hence, $(Y, \mathcal{F})$ does not have property-$(P_2)$.
\end{example}

\begin{theorem}\label{thm:pfds_P2}
     Let $\{X_i :{i\in I}\}$ be a finite collection of Banach spaces, $Y_i$ be a subspace of $X_i$ for every $i\in I$ and $1\leq p < \infty.$ Let $X=(\oplus_pX_i)_{i \in I}$ and $Y=(\oplus_pY_i)_{i \in I}.$ Then $(Y_i, \mathcal{F}_i)$ has property-$(P_2)$ for every $i\in I$, where $\mathcal{F}_i=\{F_i \in CB(X_i): rad_{Y_i}(F_i)= 1\}$  if and only if  $(Y, \mathcal{F})$ has property-$(P_2)$, where $\mathcal{F}=\{F \in \mathcal{P}(X): rad_Y(F) = 1\}$.
\end{theorem}

\begin{proof}
It is enough to prove the necessary part. Clearly, by \Cref{prop: rcp l_p}, $(Y, \mathcal{F})$ has r.c.p. Let $(G_n)$ be a  sequence in $\mathcal{F}$ and $(w_n)$ be a sequence in $Y$ such that $r(w_n, G_n) - rad_Y(G_n) \to 0,$ where $G_n= \Pi_{i \in I} G_{n,i}$ and $w_n=(w_{n,i})_{i \in I}$ for every $n \in \mathbb{N}$. Let $i \in I.$ Since 
\begin{align*}
     0 &\leq r(w_{n, i}, G_{n,i})^p- rad_{Y_i}(G_{n,i})^p\\
     & = r(w_{n, i}, G_{n,i})^p + \sum_{j\in I {\setminus} \{i\}} rad_{Y_j}(G_{n,j})^p- \left(rad_{Y_i}(G_{n,i})^p + \sum_{j\in I {\setminus}  \{i\}} rad_{Y_j}(G_{n,j})^p\right)\\
     & \leq r(w_n, G_n)^p-rad_Y(G_n)^p,
\end{align*}
we have $r(w_{n,i}, G_{n,i})- rad_{Y_i}(G_{n,i}) \to 0$ as $n \to \infty$. Since $rad_{Y_i}(G_{n,i}) \leq 1$, by assumption and \Cref{prop:P_2Equi}, there exists $z_{n,i} \in Z_{Y_i}(G_{n,i})$ for every $n \in \mathbb{N}$ such that $\|w_{n, i}-z_{n, i}\| \to 0$ as $n \to \infty.$ Now for every $n \in \mathbb N,$ consider $z_n=(z_{n,i})_{i\in I} \in Y$ and observe that, by \Cref{prop: rcp l_p}, $z_n \in Z_Y(G_n)$. Clearly, we have $\|w_n-z_n\| \to 0.$ Thus, $(Y, \mathcal{F})$ has property-$(P_2)$. 
\end{proof}

The proof of the next result follows in the similar lines of the proof of \Cref{thm:pfds_P2}.

\begin{corollary}\label{thm:pfds_USP}
    Let $\{X_i :{i\in I}\}$ be a finite collection of Banach spaces, $Y_i$ be a subspace of $X_i$ for every $i\in I$ and $1\leq p < \infty.$ Let $X=(\oplus_pX_i)_{i \in I}$ and $Y=(\oplus_pY_i)_{i \in I}.$ Then $Y_i$ is uniformly strongly proximinal on $Y_i^\prime$ for every $i\in I$, where $Y'_i=\{x_i \in X_i: d(x_i, Y_i) =1\}$  if and only if $Y$ is uniformly strongly proximinal on $Y^\prime$, where $Y'=\{x \in X: d(x, Y) =1\}$. 
\end{corollary}
    
Let $V \in CL(X)$ and $\mathcal{F} \subseteq CB(X)$. Recall that if $(V, \mathcal{F})$ has property-$(P_1),$ then the restricted center map $Z_V$ is uHsc on $(\mathcal{F}, \mathcal{H})$. Also if $(V, \mathcal{F})$ has property-$(P_2),$ then the map  $Z_V$ is uniformly Hausdorff continuous on $(\mathcal{F}, \mathcal{H})$. It is interesting to check the existence of an intermediate notion that provides the Hausdorff continuity of restricted center map.  For this, we introduce the following notion.

\begin{definition}
    Let $V \in CL(X)$ and $\mathcal{F} \subseteq CB(X).$ We say that the pair $(V, \mathcal{F})$ has property-$(lP_2)$, if $(V, \mathcal{F})$ has r.c.p. and for every  $\epsilon>0$, $F \in \mathcal{F}$ there exists $\delta>0$ such that $Z_V(G, \delta) \subseteq Z_V(G)+\epsilon B_X$ whenever  $G \in \mathcal{F}$ and $\mathcal{H}(F, G)< \delta.$
\end{definition}

Note that, in the above definition, if $\mathcal{F}$ is a collection of all singleton subsets of $X$, then property-$(lP_2)$ is equivalent to locally uniform strong proximinality \cite{InPr2017a}, i.e., $(V, \mathcal{F})$ has property-$(lP_2)$ if and only if $V$ is locally uniformly strongly proximinal on $X$. Clearly,  $(V, \mathcal{F})$ has property-$(P_2) \Rightarrow (V, \mathcal{F})$ has property-$(lP_2) \Rightarrow (V, \mathcal{F})$ has property-$(P_1)$. However, the reverse implications are not true in general (see, \Cref{eg:_p1_lp1_p2}). 

In the following result, we establish a relation between property-$(P_1)$ and property-$(lP_2).$ This relation exhibits that property-$(lP_2)$ is sufficient for the Hausdorff continuity of the restricted center map.

\begin{theorem}\label{thrm: P_1 lP_2 lHsc}
    Let $V \in CL(X)$ and $\mathcal{F} \subseteq CB(X).$ Then the following statements are equivalent.
    \begin{enumerate}
    \item $(V, \mathcal{F})$ has property-$(lP_2).$
    
    \item For every $\epsilon>0$ and $F \in \mathcal{F}$ there exists $\delta>0$ such that $Z_V(F, \delta) \subseteq Z_V(G)+ \epsilon B_X$ whenever $G \in \mathcal{F}$ and $\mathcal{H}(F,G)< \delta.$ 
    
    \item $(V, \mathcal{F})$ has property-$(P_1)$ and the map $Z_V$ is lHsc on $(\mathcal{F}, \mathcal{H})$.
    \end{enumerate}
\end{theorem}

\begin{proof}
    $(1) \Rightarrow (2)$: Let $\epsilon>0$ and $F \in \mathcal{F}$. By $(1)$, there exists $\delta>0$ such that  $Z_V(G, \delta) \subseteq Z_V(G)+ \epsilon B_X$ whenever $G \in \mathcal{F}$ and $\mathcal{H}(F,G)< \delta.$ Choose $\delta_1= \frac{\delta}{3}$. Let  $G \in \mathcal{F}$ and $\mathcal{H}(F,G)< \delta_1$.  Since for any $v \in Z_V(F, \delta_1)$ we have 
    \[ r(v, G) \leq r(v, F)+\mathcal{H}(F, G) < rad_V(F)+2\delta_1 < rad_V(G)+3\delta_1,\]
    it follows that $Z_V(F, \delta_1) \subset Z_V(G, \delta).$ Therefore, we have  $Z_V(F, \delta_1) \subseteq Z_V(G)+ \epsilon B_X$.\\
    $(2) \Rightarrow (3)$: Obvious.\\
    $(3) \Rightarrow (1)$: Let $\epsilon >0$ and $F \in \mathcal{F}.$ By our assumption, there exists $\delta' >0$ such that $Z_V(F, \delta') \subseteq Z_V(F)+\frac{\epsilon}{2} B_X$ and $Z_V(F) \subseteq Z_V(A)+\frac{\epsilon}{2} B_X$ for every $A \in \mathcal{F}$ with $\mathcal{H}(F, A)<\delta'.$ For $\delta =\frac{\delta^\prime}{3},$ we claim that $Z_V(G, \delta) \subseteq Z_V(G)+\epsilon B_X$ for any $G\in \mathcal{F}$ with $\mathcal{H}(F, G)<\delta.$ Let $v \in Z_V(G, \delta).$ Since
    \[ r(v, F) \leq r(v, G)+\mathcal{H}(F, G) < rad_V(G)+2\delta < rad_V(F)+3\delta,\]
    we have $v \in  Z_V(F,\delta^\prime).$ Thus, \[Z_V(G, \delta) \subseteq Z_V(F, \delta^\prime) \subseteq Z_V(F)+\frac{\epsilon}{2} B_X \subseteq Z_V(G)+\epsilon B_X.\]
    Hence, $(V, \mathcal{F})$ has property-$(lP_2).$
\end{proof}

\begin{corollary}
    Let $V \in CL(X)$. Then $V$  is strongly proximinal on $X$ and the map $P_V$ is lHsc on $X$ if and only if  $V$  is locally uniformly strongly proximinal on $X$.
\end{corollary}

The following example reveals that the property-$(lP_2)$ is stronger than property-$(P_1)$ but weaker to property-$(P_2).$ We say a space $X$ is uniformly rotund in every direction (in short, URED) \cite{AmZi1980}, if $X$ is uniformly rotund with respect to every one dimensional subspace of $X$. 

\begin{example}\label{eg:_p1_lp1_p2}
    \begin{enumerate}
        \item Consider $X=(\mathbb{R}^3, \|\cdot\|_c),$ where $\|x\|_c=|x_1|+(x_2^2+ x_3^2)^{\frac{1}{2}}$ for any $x=(x_1, x_2, x_3) \in \mathbb{R}^3$ and $Y= span \{(1,1,0)\}.$ Note that, from \cite[Page 193]{DeKe1983}, the metric projection $P_Y$ is not lHsc on $X.$ Thus, $Z_Y$ is not lHsc on $CB(X).$ Therefore, by \Cref{thrm: P_1 lP_2 lHsc}, $(Y,CB(X))$ does not have property-$(lP_2).$ However, by \cite[Proposition 1]{Mach1980}, $(Y, CB(X))$ has property-$(P_1).$
        \item Consider a rotund space $X$ which is not URED (see, \cite[Example 1]{Smit1978a}).  Then there exists a one dimensional subspace $Y$ of $X$ such that $X$ is not uniformly rotund with respect to $Y.$  By \cite[Proposition 1]{Mach1980}, $(Y,\mathcal{F})$ has property-$(P_1)$, where $\mathcal{F}=\{F \in K(X): rad_Y(F)=1\}.$ Since $X$ is rotund, by \cite{Gark1962}, $Z_Y(A)$ is singleton for every $A \in \mathcal{F}.$  Thus, $Z_Y$ is Hausdorff continuous on $(\mathcal{F}, \mathcal{H}).$ Therefore, by \Cref{thrm: P_1 lP_2 lHsc}, $(Y,\mathcal{F})$ has property-$(lP_2).$ However, from \Cref{lem: UR w.r.t. Y}, it follows that $(Y,\mathcal{F})$ does not have property-$(P_2).$ 
        \end{enumerate}
\end{example}

 Whenever $X$ is rotund, $V \in CC(X)$ and $\mathcal{F} \subseteq K(X),$ it follows from \cite[Lemma 1.2]{AmZi1980} and \Cref{thrm: P_1 lP_2 lHsc} that $(V, \mathcal{F})$ has property-$(P_1)$ if and only if $(V, \mathcal{F})$ has property-$(lP_2)$. However, in \Cref{eg:_p1_lp1_p2}, notice that even in finite dimensional spaces, property-$(P_1)$ and property-$(lP_2)$ are not equivalent on $CB(X).$ In the following result, we observe that property-$(P_1)$ and property-$(lP_2)$ coincide for a large class of spaces.

\begin{proposition}\label{prop: URED}
  Let $X$ be a URED space. Then the following statements hold.
        \begin{enumerate}
            \item Let $C \in CC(X)$, $\mathcal{F} \subseteq CB(X)$. Then $(C,\mathcal{F})$ has property-$(P_1)$ if and only if   $(C,\mathcal{F})$ has property-$(lP_2).$ 
            \item Let $Y$ be any finite dimensional subspace of $X$. Then $(Y, \mathcal{F}')$ has property-$(P_2)$, where $\mathcal{F}'=\{F \in CB(X): rad_Y(F)=1\}.$
        \end{enumerate}
  \end{proposition}
\begin{proof}
$(1)$: Let $(C,\mathcal{F})$ has property-$(P_1)$. By \cite{Gark1962}, $Z_C(A)$ is singleton for every $A \in \mathcal{F}.$ Therefore, by \cite[Theorem 5]{Mach1980} and \Cref{thrm: P_1 lP_2 lHsc},  $(C,\mathcal{F})$ has property-$(lP_2).$\\
$(2)$: Since $X$ is URED and $dim(Y)< \infty$, by \cite[Proposition 1]{Zizl1972}, it follows that $X$ is uniformly rotund with respect to $Y$. Therefore, by \Cref{lem: UR w.r.t. Y}, $(Y, \mathcal{F}')$ has property-$(P_2)$.
\end{proof}

In general, for hyperplane $H$ of $X$, from \cite[Proposition 2.6]{DuST2017} it follows that the notions strong proximinality, locally uniform strong proximinality and uniform strong proximinality coincide. However the following example reveals that the properties-$(P_1)$, $(lP_2)$ and $(P_2)$ do not coincide for hyperplanes.  We say that a Banach space $X$ has Kadets-Klee property if $x \in S_X$ and $(x_n)$ is a sequence in $S_X$ with $x_n \xrightarrow{w} x$, it follows that $x_n \to x.$ 
    
\begin{example}\label{eg: lP_2 P_2}
\begin{enumerate}
    \item Let $W$ be any finite dimensional space such that the  center map $Z_W$ is not Hausdorff  continuous on $CB(W)$ \cite[Example 2.5]{AmMS1982}. Consider $X=W \oplus_2 \mathbb{R}$ and $Y= W \oplus_2 \{0\}.$ Clearly, the restricted center map $Z_Y$ is not Hausdorff continuous on $CB(X)$. Therefore, by \Cref{thrm: P_1 lP_2 lHsc}, $(Y, CB(X))$ does not have property-$(lP_2).$ However, by \cite[Proposition 1]{Mach1980}, $(Y, CB(X))$ has property-$(P_1).$

    \item Consider a non uniformly rotund space $M$ which is rotund, reflexive and has  Kadets-Klee property (see, \cite[Example 1]{Smit1978a}). Let $X= M \oplus_2 \mathbb{R}$ and $H=M \oplus_2 \{0\}.$ Clearly, $X$ is rotund, reflexive and has Kadets-Klee property. Thus, by \cite[Theorem 2.3]{LPST2017} and \cite[Lemma 1.2]{AmZi1980}, it follows that $(H, K(X))$ has property-$(P_1)$ and $Z_H(A)$ is singleton for every $A \in K(X).$ Further, by \cite[Theorem 5]{Mach1980} and \Cref{thrm: P_1 lP_2 lHsc}, $(H, K(X))$ has property-$(lP_2).$ Since $X$ is not uniformly rotund with respect to $H,$ by \Cref{lem: UR w.r.t. Y}, $(H, \mathcal{F})$ does not have  property-$(P_2)$, where $\mathcal{F}=\{F \in K(X): rad_H(F) \leq 1\}$.
\end{enumerate}
\end{example}

\Cref{eg:_p1_lp1_p2,eg: lP_2 P_2} further reveals that, in general,  properties-$(lP_2)$ and $(P_2)$ do not coincide in infinite dimensional spaces even when the subspaces are finite dimensional or finite co-dimensional.  However, the following result illustrates that in  finite dimensional spaces, the properties-$(lP_2)$ and $(P_2)$ coincide. 

\begin{proposition}\label{prop: compact lP_2}
    Let $V \in CL(X)$ and $\mathcal{F}$ be a compact subset of $(CB(X), \mathcal{H})$. Then  $(V, \mathcal{F})$ has property-$(lP_2)$ if and only if $(V, \mathcal{F})$ has property-$(P_2)$.
\end{proposition}

\begin{proof}
    It is enough to prove the necessary part. Let $\epsilon>0$. By assumption, for every $F \in \mathcal{F}$ there exists $\delta_F>0$ such that $Z_V(G, \delta_F) \subseteq Z_V(G)+ \epsilon B_X$ whenever $G \in \mathcal{F}$ and $\mathcal{H}(F,G)< \delta_F.$ Since $\{ B(F, \delta_F): F \in \mathcal{F}\}$ is an open cover for  $\mathcal{F}$, there exist $F_1, F_2, \dots, F_n \in \mathcal{F}$ such that $\mathcal{F} \subseteq  \cup_{i=1}^{n} B(F_i, \delta_{F_i})$. Choose $\delta= \min\{\delta_{F_i}: 1 \leq i \leq n\}.$ Let $F \in \mathcal{F}.$ Then $F \in B(F_j, \delta_{F_j})$ for some $1 \leq j \leq n.$ Therefore, $Z_V(F, \delta) \subseteq Z_V(F, \delta_{F_j}) \subseteq Z_V(F)+ \epsilon B_X.$ Hence the proof. 
\end{proof}

\begin{corollary}\label{cor: finite lP2 P_2}
Let $X$ be a finite dimensional space and  $Y$ be a subspace of $X$. Then $(Y, CB(X))$ has property-$(lP_2)$ if and only if $(Y, \mathcal{F})$ has property-$(P_2)$, where $\mathcal{F}=\{F \in CB(X): rad_Y(F)=1\}$.
\end{corollary}

\begin{proof}
   Let $(Y, CB(X))$ has property-$(lP_2)$. By \Cref{prop:P_2Equi}, it is enough to prove that $(Y, \mathcal{F}')$ has property-$(P_2)$, where $\mathcal{F}'=\{F \in CB(X): rad_Y(F)= r(0,F)=1\}$. Observe that $\mathcal{F}'$ is a closed subset of $CB(B_X).$ Since $B_X$ is compact, by \cite[Theorem 3.2.4]{Beer1993}, $(CB(B_X), \mathcal{H})$ is compact subset of $(CB(X), \mathcal{H})$. Therefore, $\mathcal{F}'$ is compact subset of $(CB(X), \mathcal{H})$. Thus, by \Cref{prop: compact lP_2}, $(Y, \mathcal{F}')$ has property-$(P_2)$. The converse is easy to verify using \Cref{prop:P_2Equi}. Hence the proof.
\end{proof}

The following corollary is a consequence of \Cref{lem: QUR,thrm: P_1 lP_2 lHsc}, and \Cref{cor: finite lP2 P_2}. The result also reveals that the lower Hausdorff semi-continuity of restricted center map is sufficient for the quasi uniform rotundity whenever the space is of finite dimension.   
\begin{corollary}
    Let $X$ be a finite dimensional space and  $Y$ be a subspace of $X$. If $Z_Y$ is lHsc on $(CB(X), \mathcal{H})$, then $Z_Y$ is uniformly Hausdorff continuous on $(\mathcal{F}, \mathcal{H})$, where $\mathcal{F}=\{F \in CB(X): rad_Y(F)\leq 1\}$. In particular, $X$ is quasi uniformly rotund with respect to $Y$.  
\end{corollary}

In view of \Cref{thm:pds_P1,thm:pfds_P2} and  \Cref{eg: p2}, it is reasonable to investigate the stability of property-$(lP_2)$ under $\ell_p$-direct sum. To examine this, it is enough to check, by \Cref{thrm: P_1 lP_2 lHsc}, the stability of continuity properties of the restricted center map.

\begin{theorem}\label{thrm: lHsc P_1}
Let $\{X_i: {i\in \mathbb{N}}\}$ be a collection of Banach spaces, $Y_i$ be a  subspace of $X_i$ such that $(Y_i,CB(X_i))$ has r.c.p. for every $i\in \mathbb{N}$ and $1\leq p < \infty.$ Let $X=(\oplus_pX_i)_{i \in  \mathbb{N}}$ and $Y=(\oplus_pY_i)_{i \in  \mathbb{N}}$. Then $Z_{Y_i}$ is lHsc (respectively, uHsc) on $CB(X_i)$  for every $i\in \mathbb{N}$ if and only if  $Z_Y$ is lHsc (respectively, uHsc) on $\mathcal{P}(X)$.
\end{theorem}

\begin{proof}
We present the proof for lHsc, the proof of uHsc follows in similar lines. Let $Z_{Y_i}$ be  lHsc on $CB(X_i)$ for every $i \in \mathbb{N}.$ Let $F= \Pi_{i \in \mathbb{N}} F_i \in \mathcal{P}(X)$,  $G_n = \Pi_{i \in \mathbb{N}} G_{n,i} \in \mathcal{P}(X)$ for every $n \in \mathbb{N}$ satisfying $\mathcal{H}(G_n,F) \to 0$ and $(y_n)$ be a sequence in $Z_Y(F)$ where $y_n= (y_{n,i})_{i \in \mathbb{N}}$. We need to prove that $d(y_n, Z_Y(G_n)) \to 0.$ Let $i \in \mathbb{N}$. Observe that $\mathcal{H}(G_{n,i},F_i) \leq \mathcal{H}(G_n, F)$ holds  for every $n \in \mathbb{N}$, hence $\mathcal{H}(G_{n,i},F_i) \to 0 $ as $n \to \infty$. Since, by \Cref{prop: rcp l_p}, $y_{n,i} \in Z_{Y_i}(F_i)$, it follows from the  assumption that there exists $w_{n,i} \in Z_{Y_i}(G_{n,i})$ for every $n \in \mathbb{N}$ such that $\|y_{n,i}-w_{n,i}\| \to 0$ as $n \to \infty$. Note that for every $n \in \mathbb{N}$, by \Cref{prop: rcp l_p}, $w_n \in Z_{Y}(G_n)$, where $w_n = (w_{n,i})_{i \in \mathbb{N}}.$  Now, it is enough to prove that $\|y_n-w_n\| \to 0$. For this, let $\epsilon>0$. Choose $j_0 \in \mathbb{N}$ such that,
\begin{itemize}
    \item $\sum_{i>j_0}\|y_{n,i}\|^p < \epsilon ^p$ for all $n \in \mathbb{N}$ (by \Cref{lem:unimini});
    
    \item $\sum_{i> j_0}rad_{Y_i}(F_i)^p< \frac{\epsilon^p}{2^{p+2}}$ (by \Cref{prop:RadiusEquality_p}) and 
    
    \item  $\sum_{i> j_0}\|z_i\|^p < \left( \frac{\epsilon}{4} \right)^p$, for some fixed $z=(z_i)_{i \in \mathbb{N}} \in F$
\end{itemize}
holds. Since $\mathcal{H}(G_n, F) \to 0$ and $\mathcal{H}(G_{n,i},F_i) \to 0$ for all $i \in \mathbb{N}$,  choose $n_0 \in \mathbb{N}$ such that for all $n \geq n_0,$ we have 

\begin{itemize}
     \item $|rad_Y(G_n)^p- rad_Y(F)^p| < \frac{\epsilon^p}{2^{p+2}}$; \hfill{$(*)$}
     
    \item $\mathcal{H}(G_n,F) < \frac{\epsilon}{8}$; \hfill{$(**)$}
    
    \item  $\sum_{i \leq j_0}| rad_{Y_i}(G_{n,i})^p-rad_{Y_i}(F_i)^p|< \frac{\epsilon^p}{2^{p+2}}$ and 
    
    \item $\sum_{i \leq j_0}\|y_{n,i}-w_{n,i}\|^p < (2\epsilon)^p$.
\end{itemize}
Let $n \geq n_0.$ 
Therefore, by $(*)$ and using  \Cref{prop:RadiusEquality_p},
\begin{align*}
\sum_{i> j_0}rad_{Y_i}(G_{n,i})^p &< \sum_{i \leq j_0} rad_{Y_i}(F_i)^p- \sum_{i \leq j_0} rad_{Y_i}(G_{n,i})^p+ \sum_{i > j_0} rad_{Y_i}(F_i)^p+ \frac{\epsilon^p}{2^{p+2}}\\
& <  \frac{\epsilon^p}{2^{p+2}}+  \frac{\epsilon^p}{2^{p+2}}+  \frac{\epsilon^p}{2^{p+2}}\\
& < \left(\frac{\epsilon}{2}\right)^p.
\end{align*}
Since, by $(**)$, $F \subseteq G_n+ \frac{\epsilon}{8}B_X$ holds, choose $u_n=(u_{n,i})_{i \in \mathbb{N}} \in G_n$ such that $\|u_n-z\| \leq \frac{\epsilon}{8}$. Thus, 
\[
\left(\sum_{i> j_0} \|u_{n,i}\|^p\right)^{\frac{1}{p}} \leq \left(\sum_{i> j_0} \|u_{n,i}-z_i\|^p\right)^{\frac{1}{p}}+\left(\sum_{i> j_0} \|z_{i}\|^p\right)^{\frac{1}{p}} < \|u_n-z\| + \frac{\epsilon}{4} < \frac{\epsilon}{2}.
\]
Therefore, it follows that
\begin{align*}
\left(\sum_{i>j_0}\|w_{n,i}\|^p \right)^{\frac{1}{p}} &\leq \left(\sum_{i> j_0} r(w_{n,i},G_{n,i})^p \right)^{\frac{1}{p}}+ \left(\sum_{i>j_0}\|u_{n,i}\|^p \right)^{\frac{1}{p}}\\
 & =  \left(\sum_{i> j_0} rad_{Y_i}(G_{n,i})^p \right)^{\frac{1}{p}}+ \left(\sum_{i>j_0}\|u_{n,i}\|^p \right)^{\frac{1}{p}}\\ 
 & < \frac{\epsilon}{2}+ \frac{\epsilon}{2}\\
 & = \epsilon.
\end{align*}
 Thus,
  \begin{align*}
     \|y_n-w_n\|^p &= \sum_{i \leq j_0}\|y_{n,i}-w_{n,i}\|^p+ \sum_{i> j_0}\|y_{n,i}-w_{n,i}\|^p\\
     & \leq  \sum_{i \leq j_0}\|y_{n,i}-w_{n,i}\|^p+ \left(\left(\sum_{i> j_0}\|y_{n,i}\|^p\right)^{\frac{1}{p}} + \left(\sum_{i> j_0}\|w_{n,i}\|^p\right)^{\frac{1}{p}}\right)^p\\
     & <  (2\epsilon)^{p}+(2\epsilon)^{p}.
  \end{align*}
  
   \noindent Hence, $\|y_n-w_n\| \to 0.$ This completes the proof of necessary part. 
   
   \par To prove the converse, let $j \in \mathbb{N}$ and $F_j \in CB(X_j).$ Consider $F= \Pi_{i \in \mathbb{N}} F_i \in \mathcal{P}(X)$, where $F_i=\{0\}$ for all $i \in \mathbb{N} {\setminus} \{j\}.$ Since $Z_Y$ is lHsc on $\mathcal{P}(X),$ there exists $\delta>0$ such that $Z_Y(F) \subseteq Z_Y(E) + \epsilon B_X$, for every $E \in \mathcal{P}(X)$ with $\mathcal{H}(F,E) < \delta.$  Let $G_j \in CB(X_j)$ such that $\mathcal{H}(F_j,G_j) < \delta.$ Now, we need to show that $Z_{Y_j}(F_j) \subseteq Z_{Y_j}(G_j)+ \epsilon B_{X_j}$. Consider $G=\Pi_{i \in \mathbb{N}} G_i \in \mathcal{P}(X)$, where $G_i=\{0\}$ for every $i \in \mathbb{N} {\setminus} \{j\}.$ Since $\mathcal{H}(F,G)= \mathcal{H}(F_j,G_j)$, by assumption, $Z_Y(F) \subseteq Z_Y(G)+ \epsilon B_X.$ Further, by \Cref{prop: rcp l_p}, $Z_{Y_j}(F_j) \subseteq Z_{Y_j}(G_j)+ \epsilon B_{X_j}.$ Thus, $Z_{Y_j}$ is lHsc on $CB(X_j)$. Hence the proof. 
\end{proof}

The following result on the stability of the continuity properties of the metric projection map is an immediate consequence of the preceding theorem. 

\begin{corollary}
Let $\{X_i: {i\in \mathbb{N}}\}$ be a collection of Banach spaces, $Y_i$ be a  proximinal subspace of $X_i$ for every $i\in \mathbb{N}$ and $1\leq p < \infty.$ Let $X=(\oplus_pX_i)_{i \in  \mathbb{N}}$ and $Y=(\oplus_pY_i)_{i \in  \mathbb{N}}$. Then $P_{Y_i}$ is lHsc (respectively, uHsc) on $X_i$  for every $i\in \mathbb{N}$ if and only if  $P_Y$ is lHsc (respectively, uHsc) on $X$.
\end{corollary}

From \Cref{thm:pds_P1,thrm: P_1 lP_2 lHsc,thrm: lHsc P_1}, it is easy to verify the stability of property-$(lP_2),$ as shown in the result that follows. Further, as an outcome, we obtain the stability of locally uniform strong proximinality.

\begin{theorem}\label{thm:pds_lP1}
    Let $\{X_i: {i\in \mathbb{N}}\}$ be a collection of Banach spaces, $Y_i$ be a  subspace of $X_i$ for every $i\in \mathbb{N}$ and $1\leq p < \infty.$ Let $X=(\oplus_pX_i)_{i \in  \mathbb{N}}$ and  $Y=(\oplus_pY_i)_{i \in  \mathbb{N}}$. Then $(Y_i, CB(X_i))$ has property-$(lP_2)$ for every $i\in \mathbb{N}$ if and only if  $(Y, \mathcal{P}(X))$ has property-$(lP_2).$ 
\end{theorem}

\begin{corollary}\label{cor:pds_lusp}
 Let $\{X_i: {i\in \mathbb{N}}\}$ be a collection of Banach spaces, $Y_i$ be a  subspace of $X_i$ for every $i\in \mathbb{N}$ and $1\leq p < \infty.$ Let $X=(\oplus_pX_i)_{i \in \mathbb{N}},$ and $Y=(\oplus_pY_i)_{i \in \mathbb{N}}.$ Then  $Y_i$ is locally uniformly strongly proximinal on $X_i$ for every $i\in \mathbb{N}$ if and only if $Y$ is locally uniformly strongly proximinal on $X$. 
\end{corollary}

\section*{Declarations}
\textbf{Funding:} No funding was received for conducting this study.\\
	\textbf{Data availability:} Data sharing not applicable to this article.\\
		\textbf{Conflict of interest:} The authors have no conflict of interest to declare.


\begin{thebibliography}{10}

\bibitem{AmMS1982}
D.~Amir, J.~Mach, and K.~Saatkamp.
\newblock Existence of {C}hebyshev centers, best {$n$}-nets and best compact
  approximants.
\newblock {\em Trans. Amer. Math. Soc.}, 271(2):513--524, 1982.

\bibitem{AmZi1980}
D.~Amir and Z.~Ziegler.
\newblock Relative {C}hebyshev centers in normed linear spaces. {I}.
\newblock {\em J. Approx. Theory}, 29(3):235--252, 1980.

\bibitem{BLLN2008}
P.~Bandyopadhyay, Y.~Li, B.~L. Lin, and D.~Narayana.
\newblock Proximinality in {B}anach spaces.
\newblock {\em J. Math. Anal. Appl.}, 341(1):309--317, 2008.

\bibitem{Beer1993}
G.~Beer.
\newblock {\em Topologies on closed and closed convex sets}, volume 268 of {\em
  Mathematics and its Applications}.
\newblock Kluwer Academic Publishers Group, Dordrecht, 1993.

\bibitem{CaCH1973}
J.~R. Calder, W.~P. Coleman, and R.~L. Harris.
\newblock Centers of infinite bounded sets in a normed space.
\newblock {\em Canadian J. Math.}, 25:986--999, 1973.

\bibitem{DeKe1983}
F.~Deutsch and P.~Kenderov.
\newblock Continuous selections and approximate selection for set-valued
  mappings and applications to metric projections.
\newblock {\em SIAM J. Math. Anal.}, 14(1):185--194, 1983.

\bibitem{DuST2017}
S.~Dutta, P.~Shunmugaraj, and V.~Thota.
\newblock Uniform strong proximinality and continuity of metric projection.
\newblock {\em J. Convex Anal.}, 24(4):1263--1279, 2017.

\bibitem{Gark1962}
A.~L. Garkavi.
\newblock On the optimal net and best cross-section of a set in a normed space.
\newblock {\em Izv. Akad. Nauk SSSR Ser. Mat.}, 26(1):87--106, 1962.

\bibitem{InPr2017a}
V.~Indumathi and N.~Prakash.
\newblock Local uniform strong proximinality of the space of continuous
  vector-valued functions.
\newblock {\em Numer. Funct. Anal. Optim.}, 38(8):1014--1023, 2017.

\bibitem{LaNa2006}
S.~Lalithambigai and D.~Narayana.
\newblock Semicontinuity of metric projections in {$c_0$}-direct sums.
\newblock {\em Taiwanese J. Math.}, 10(5):1245--1259, 2006.

\bibitem{LPST2017}
S.~Lalithambigai, T.~Paul, P.~Shunmugaraj, and V.~Thota.
\newblock Chebyshev centers and some geometric properties of {B}anach spaces.
\newblock {\em J. Math. Anal. Appl.}, 449(1):926--938, 2017.

\bibitem{LuSZ2016}
Z.~Luo, L.~Sun, and W.~Zhang.
\newblock A remark on the stability of approximative compactness.
\newblock {\em J. Funct. Spaces}, pages Art. ID 2734947, 5, 2016.

\bibitem{Mach1980}
J.~Mach.
\newblock Continuity properties of {C}hebyshev centers.
\newblock {\em J. Approx. Theory}, 29(3):223--230, 1980.

\bibitem{Rao2016}
T.~S. S. R.~K. Rao.
\newblock Simultaneously proximinal subspaces.
\newblock {\em J. Appl. Anal.}, 22(2):115--120, 2016.

\bibitem{Smit1978a}
M.~A. Smith.
\newblock Some examples concerning rotundity in {B}anach spaces.
\newblock {\em Math. Ann.}, 233(2):155--161, 1978.

\bibitem{Vese2017}
L.~Vesel\'{y}.
\newblock Quasi uniform convexity---revisited.
\newblock {\em J. Approx. Theory}, 223:64--76, 2017.

\bibitem{Zizl1972}
V.~Zizler.
\newblock Uniform extension of linear functionals.
\newblock {\em \v{C}asopis P\v{e}st. Mat.}, 97:379--385, 420, 1972.

\end{thebibliography}
\end{document}